\documentclass[12pt]{amsart}
\usepackage{amsfonts}
\usepackage{amsfonts,latexsym,rawfonts,amsmath,amssymb,amsthm}
\usepackage[plainpages=false]{hyperref}
\usepackage{cite}

\usepackage{graphicx}

\RequirePackage{color}

\numberwithin{equation}{section}

\newcommand{\beq}{\begin{equation}}
\newcommand{\eeq}{\end{equation}}
\newcommand{\beqs}{\begin{eqnarray*}}
\newcommand{\eeqs}{\end{eqnarray*}}
\newcommand{\beqn}{\begin{eqnarray}}
\newcommand{\eeqn}{\end{eqnarray}}
\newcommand{\beqa}{\begin{array}}
\newcommand{\eeqa}{\end{array}}

\newtheorem{definition}{Definition}
\newtheorem{lemma}{Lemma}
\newtheorem{remark}{Remark}
\newtheorem{proposition}{Proposition}

\newtheorem{theorem}{Theorem}
\newtheorem{example}{Example}
\newtheorem{corollary}{Corollary}

\title  {sensitivity, transitivity and chaos in non-autonomous discrete systems}
\begin{document}




\bibliographystyle{plain}

\maketitle

\baselineskip=15.8pt
\parskip=3pt

\centerline {\bf   \small Hongbo Zeng}
\centerline {School of Mathematics and Statistics, Changsha University of Science and Technology;}
\centerline {Hunan Provincial Key Laboratory of Mathematical Modeling and Analysis in Engineering,}
\centerline {Changsha 410114, Hunan, China}

\vskip20pt

\noindent {\bf Abstract}:
In this paper, we study properties of sensitivity, transitivity and chaos for non-autonomous discrete systems(NDS). Firstly, we present some different sufficient conditions for NDS to be chaotic.  Then, we relate the transitivity with the sensitivity of NDS and give several sufficient conditions for NDS to be sensitive. We obtain that transitivity and dense periodic points imply sensitivity, and that transitive system is either sensitive or almost equicontinuous. The results improve and extend some existing ones. Besides, we give some examples to show that there is a significant difference between the theory of ADS and the theory of NDS. We get that almost periodic point and minimal point do not imply each other and that two definitions of minimal system are not equivalent for non-autonomous discrete systems. Finally, we introduce and study weakly sensitivity for non-autonomous discrete systems.

 \vskip20pt

\vskip20pt

\baselineskip=15.8pt
\parskip=3pt


\maketitle

\baselineskip=15.8pt
\parskip=3.0pt




\section{Introduction}

Dynamical systems theory is an effective mathematical mechanism which describes the time dependence of a point in a geometric space and has remarkable connections with different areas of mathematics such as topology and number theory. It is used to deal with the complexity, instability, or chaoticity in the real world, such as in meteorology, ecology, celestial mechanics, and other natural sciences. In recent years, more and more scholars have begun to devote themselves to the research in topological dynamical systems, and have achieved many significant results (see \cite{q12,q7}). The chaos theory is one of the most important components of dynamical systems, which was first strictly defined by Li and Yorke in 1975 \cite{a1}. Since then, different people from different fields gave different definitions of chaos under their understanding of the subject. In 1986, Devaney proposed the widely accepted definition of chaos \cite{a3}. Later, Banks et al found that in the three conditions defining Devaney chaos, topological transitivity and dense periodic points together imply sensitivity \cite{q8}. Moreover, Kato defined the everywhere chaotic in \cite{a4}.  Sensitive dependence on initial conditions(briefly, sensitivity), first defined in \cite{q9}, is also one of the most remarkable components of dynamical systems theory, which is closely linked to different variants of mixing. It characterizes the unpredictability in chaotic phenomena and is an integral part of different types of chaos. Afterwards, Li-Yorke sensitivity \cite{q10} and some stronger forms of sensitivity (including cofinite sensitivity, multisensitivity, and syndetic sensitivity) \cite{q11} were successively proposed. Besides, as we all know, topological transitivity (shortly, transitivity) has been an eternal topic in the study of topological dynamical systems, which is another crucial measure of system complexity \cite{q12,q7}.

\par
   As a natural extension of autonomous discrete dynamical systems (Abbrev. ADS), non-autonomous discrete dynamical systems (NDS) are an important part of topological
dynamical systems. Compared with classical dynamical systems (ADS), NDS can
describe various dynamical behaviors more flexibly and conveniently. Indeed, most of the natural phenomena, whose behavior is influenced by external forces, are time dependent external forces. As a result, many of the methods, concepts and results of autonomous dynamical systems are not applicable. Therefore, there is a strong need to study and develop the theory of time variant dynamical systems, that is, non-autonomous dynamical systems. As a consequence, the techniques used in this context are, in general, different from those used for autonomous systems and make this discipline of great interest. In such systems the trajectory of a point is given by successive application of different maps. These systems are related to the theory of difference equations, and in general, they provide a more adequate framework for the study of natural phenomena that appear in biology, physics, engineering, etc. Meanwhile, the dynamics of non-autonomous discrete systems has became
an active research area, obtaining results on topological entropy, sensitivity, mixing properties, chaos, and other properties.

The notion of non-autonomous dynamical system was introduced by Kolyada and Snoha \cite{a14} in 1996. Since then, the study of chaos and complexity of non-autonomous dynamical systems has seen remarkable increasing interest of many researchers, see \cite{q16,q4} and the references therein. Chaos and sensitivity for NDS were introduced by Tian and Chen  in 2006 \cite{q16}. Then, \cite{q19} proposed the concept of Devaney chaos for NDS, and asked as an open problem that whether or not the Bank theorem (topological transitivity and dense periodicity imply sensitivity)  can be generalized from ADS to NDS. However, Miralles et al \cite{q14} claimed that it is not true in general and proved that it is true if NDS is uniformly convergent. Then Yang and Li gave another sufficient condition for this result on
a finitely generated NDS\cite{q2}. Cnovas\cite{a19} studied the limit behaviour of sequence. Wu and Zhu \cite{1a} proved that the chaotic behaviour
of NDS is inherited under iterations. Recently, sensitivity and stronger forms of sensitivity are widely studied in non-autonomous systems \cite{q5}.
The interested reader in chaos, transitivity and sensitivity for NDS might consult \cite{q23,q24,q25}.

Motivated by the above results, this paper is devoted to further study the properties of chaos, transitivity and sensitivity in non-autonomous discrete systems.

\par
   This paper is organized as follows. In Section 2, we will first state some preliminaries, definitions and some lemmas. The main conclusions will be given in Section 3. We show that for compact and commutative non-autonomous discrete system, topologically transitivity and fixed point imply Li-Yorke chaos. And we prove that weakly
mixing implies Kato's chaos. Then we give some sufficient conditions under which transitivity and dense periodic points imply sensitivity. We also show that for compact and finitely generated non-autonomous discrete system, if
$f_{1,\infty}$ is transitive and commutative, then it is either sensitive or almost equicontinuous. Furthermore, we give several examples to show that the dynamical properties in NDS are quite different from what in ADS. The example shows that the Theorem 5.1 in \cite{q5} is not true. We also explore the relationship of sensitivity between ADS and NDS for periodic NDS. Finally, Section 4 posts several open problems for future research.

\section{Preparations and lemmas}
In this section, we mainly give some different concepts of chaos, sensitivity and transitivity for NDS (see,for example,\cite{q14,q5,q18}) and some lemmas.

Assume that $\mathbb{N}=\{1,2,3,...\}$ and $\mathbb{Z_+}=\{0,1,2,3,...\}$. Let $(X,d)$ be a metric
space and $f_n : X \rightarrow X$ be a sequence of continuous functions, where $n\in \mathbb{N}$. An non-autonomous discrete
dynamical systems (NDS) is a pair $(X, f_{ 1,\infty} )$ where $f_{ 1,\infty}= ( f_n )_{n=1}^\infty$. Given $i,n\in \mathbb{N}$.
Define the composition
$f_i^n:= f_{i+(n-1)} \circ \cdot\cdot\cdot \circ f_i$ ,
and usual $f_i^0=id_X$. In particular, if $f_n = f$ for all $n \in \mathbb{N}$, the pair $(X, f_{ 1,\infty} )$ is
just the classical discrete system $(X, f)$. The orbit of
a point $x$ in $X$ is the set
$orb(x, f_{ 1,\infty}):=\{x,f_1^1(x),f_1^2(x),...,f_1^n(x),...\}$,
which can also be described by the difference equation $x_0 = x$ and $x_{n+1} = f_n (x_n )$.  Denote $f_1^{-n}=(f_1^{n})^{-1}=f_1^{-1}\circ f_2^{-1}\circ\cdots\circ f_n^{-1}$.
We always suppose that $(X, f_{ 1,\infty} )$  is a non-autonomous discrete system and that all the maps are continuous from $X$ to $X$ in the following context. Denote $B(x,\varepsilon)$ the open ball of radius $\varepsilon>0$ with center $x$ and $id$ the identity map on $X$.

\begin{definition}
 A point $x\in X$ is called an equicontinuity point for $f_{ 1,\infty}$ if $f_{ 1,\infty}$ is equicontinuous at $x$, this means that for any $\varepsilon>0$, there is $\delta>0$ such that for any $y\in B(x,\delta)$, $d(f_1^n(x),f_1^n(y))<\varepsilon$ for any $n\in \mathbb{Z_+}$. We denote by $Eq(f_{ 1,\infty})$ the set of equicontinuity points of $f_{ 1,\infty}$. The non-autonomous discrete system $(X,f_{1,\infty})$ is said to be almost equicontinuous if it has a dense set of equicontinuity points, and it is called equicontinuous if $Eq(f_{ 1,\infty})=X$.
\end{definition}

\begin{definition}
 A point $x\in X$ is a recurrent point of $f_{ 1,\infty}$ if there exist an increasing sequence $\{n_i\}$ of positive integers such that
$$\lim_{i\rightarrow\infty}f_1^{n_i}(x)=x.$$ Denote the set of all  recurrent points of $f_{ 1,\infty}$ by $R(f_{ 1,\infty})$.
We denote by $int(A)$ the interior of $A$ and by $\overline{A}$ its closure. A point $x\in X$ is said to be nonrecurrent if it is not recurrent, that is, $x\notin \overline{\cup_{i=1}^\infty f_1^i(x)}$.
\end{definition}

\begin{definition}
The non-autonomous system $(X, f_{ 1,\infty} )$  is said to be (topologically) transitive, if for any two non-empty subset $U,V\subseteq X$, there exists $n\in \mathbb{N}$ such that $f_1^n(U)\cap V\neq\emptyset$. The non-autonomous system $(X, f_{ 1,\infty} )$  is said to be weakly mixing, if for any four non-empty subset $U_1,U_2,V_1,V_2\subseteq X$, there exists $n\in \mathbb{N}$ such that $f_1^n(U_1)\cap V_1\neq\emptyset$ and $f_1^n(U_2)\cap V_2\neq\emptyset$. The non-autonomous system $(X, f_{ 1,\infty} )$  is said to be mixing, if for any two non-empty subset $U,V\subseteq X$, there exists $N\in \mathbb{N}$ such that $f_1^n(U)\cap V\neq\emptyset$ for any $n\geq N$.  A point $x\in X$ is said to be transitive if the orbit of $x$ is dense in $X$. For any two non-empty subsets $U,V\subseteq X$, denote $N_{f_{ 1,\infty}}(U,V)=\{n\in \mathbb{N} \mid f_1^n(U)\cap V\neq\emptyset\}$.

\end{definition}

\begin{definition}
 A non-empty subset $U\subseteq X$ is called an invariant set of $f_{ 1,\infty}$ if $f_i(U)\subseteq U$ for any $i\in \mathbb{N}$. A set $U$ is called minimal subset of $f_{ 1,\infty}$ if it is a non-empty, invariant closed subset of $X$ and no proper subset of $U$ is non-empty, invariant and closed. A point $x\in X$ is called a minimal point if its closure orbit is a minimal subset of $f_{ 1,\infty}$. (M1) The non-autonomous system $(X, f_{ 1,\infty} )$  is said to be minimal if $X$ has no non-empty, invariant, proper and closed subset.
 (M2) The non-autonomous system $(X, f_{ 1,\infty} )$  is said to be minimal if all the points of $X$ are transitive. A non-empty subset $U\subseteq X$ is called an inversely invariant set of $f_{ 1,\infty}$ if $f_i^{-1}(U)\subseteq U$ for any $i\in \mathbb{N}$. The non-autonomous system $(X, f_{ 1,\infty} )$  is said to be feeble open if $int(f_i(U))\neq\emptyset$ for any nonempty open set $U\subseteq X$ and any $i\in \mathbb{N}$.
\end{definition}

\begin{remark}
There are two different definitions of minimal non-autonomous system $(X, f_{ 1,\infty} )$. The definition of M1 can be found in \cite{q5} and M2 can be found in \cite{q14}. We will show that M1 implies M2 in the following theorem \ref{minimal1} and that  M2 need not imply M1 in the following example \ref{minimal2}.
\end{remark}

\begin{definition}
A family $f_{ 1,\infty}$ is said to be commutative(abelian), if each of its member commutes with every other member of the family, that is, $f_i\circ f_j=f_j\circ f_i$ for any $i,j\in \mathbb{N}$.
\end{definition}

\begin{definition}
The non-autonomous system $(X, f_{ 1,\infty} )$  is said to be sensitive, if there is $\delta>0$ such that for any nonempty open set $U\subset X$, there exist $x,y\in U$ and $n\in \mathbb{N}$ such that $d(f_1^n(x),f_1^n(y))>\delta$. The non-autonomous system $(X, f_{ 1,\infty} )$  is said to be accessible, if for any $\varepsilon>0$ and any nonempty open set $U,V\subset X$, there exist $x\in U,y\in V$ and $n\in \mathbb{N}$ such that $d(f_1^n(x),f_1^n(y))<\varepsilon$. The non-autonomous system $(X, f_{ 1,\infty} )$ is said to be Kato's chaotic, if it is sensitive and accessible.

\end{definition}
Denote $N_{f_{1,\infty}}(U,\delta)=\{n\in\mathbb{N}: \text{there exist} \ x,y\in U \ \text{such that } d(f_1^n(x),f_1^n(x))<\delta \}$ for any nonempty open set $U$ of $X$. The upper density of $A\subseteq\mathbb{Z_+}$ is defined by $\bar{d}(A):=limsup_{n\rightarrow\infty}\frac{\sharp\{A\cap\{0,1,...,n-1\}\}}{n}$, where $\sharp A$ denotes the cardinality of the set $A$.
\begin{definition}
A non-autonomous system $(X, f_{ 1,\infty} )$ is said to be multi-sensitive, if there exists $\delta>0$ such that for any $m\in \mathbb{N}$ and any nonempty open sets $U_1,U_2,...,U_m\subset X$, $\bigcap_{i=1}^mN_{f_{1,\infty}}(U_i,\delta)\neq\emptyset$, where $\delta$ is called constant of sensitivity. A set $A\subseteq \mathbb{Z_+}$ is called syndetic if there exists a positive integer $M$ such that $\{i,i+1,...,i+M\}\cap A\neq\emptyset$ for every $i\in \mathbb{Z_+}$, i.e. it has bounded gaps. A set $A\subseteq \mathbb{Z_+}$ is called cofinite if there exists $N\in \mathbb{Z_+}$ such that $A\supseteq[N,\infty]\cap \mathbb{Z_+}$. A non-autonomous system $(X, f_{ 1,\infty} )$ is said to be syndetically sensitive(cofinitely sensitive, ergodically sensitive, respectively), if there exists $\delta>0$ such that for any nonempty open set $U\subset X$, $N_{f_{1,\infty}}(U,\delta)$ is a syndetic set(cofinitely set, set with positive upper density, respectively).
\end{definition}

\begin{definition}
 A set $A\subseteq \mathbb{Z_+}$ is called thick if it contains arbitrarily long runs of positive integers, that is, for any $p\in \mathbb{N}$, there exists some $n\in \mathbb{N}$ such that $\{n,n+1,...,n+p\}\subseteq A$. A non-autonomous system $(X, f_{ 1,\infty} )$ is said to be thickly sensitive, if there exists $\delta>0$ such that for any nonempty open set $U\subset X$, $N_{f_{1,\infty}}(U,\delta)$ is thick. A set $A\subseteq \mathbb{Z_+}$ is thickly  syndetic if $\{n\in \mathbb{N}\mid n+j\in A, \text{for } 0\leq j\leq k\}$ is syndetic for every $k\in\mathbb{N}$. A non-autonomous system $(X, f_{1,\infty} )$ is said to be thickly syndetically sensitive, if there exists $\delta>0$ such that for any nonempty open set $U\subset X$, $N_{f_{1,\infty}}(U,\delta)$ is thickly syndetic.
\end{definition}

\begin{definition}
 A point $x\in X$ is called $k$-periodic if $f_i^{kn}(x)=x$ for any $n\in \mathbb{N}$.  A point $x\in X$ is called fixed point if it is $1$-periodic.  A point $x\in X$ is called almost periodic point if for any $\varepsilon>0$, the set $\{n\in \mathbb{N} \mid d(f_1^n(x),x)<\varepsilon\}$ is syndetical.
\end{definition}

\begin{definition}
A non-autonomous system $(X, f_{ 1,\infty} )$ is said to be finitely generated, if there exists a finite set $F$ of continuous self maps on $X$ such that each $f_i$ of $f_{ 1,\infty}$ belongs to $F$.
\end{definition}

\begin{definition}
 If there exists an uncountable subset $S\subseteq X$ such that for any different points $x,y\in S$  we have
$$\liminf_{n\rightarrow\infty} d(f_1^{n}(x), f_1^{n}(y))=0,\ \limsup_{n\rightarrow\infty} d(f_1^{n}(x), f_1^{n}(y))>0,$$
then $f_{ 1,\infty}$ is said to be Li-Yorke chaos.
\end{definition}

\begin{definition}
 A non-autonomous discrete system $f_{ 1,\infty}$ is called Li-Yorke sensitive if there exists some $\delta>0$ such that for any $x\in X$ and any $\varepsilon>0$, there is $y\in B(x,\varepsilon)$ satisfying
 $$\liminf_{n\rightarrow\infty}d(g^n(x),g^n(y))=0 \text{   and   } \limsup_{n\rightarrow\infty}d(g^n(x),g^n(y))\geq\delta.$$
\end{definition}

\begin{lemma}\label{yinli1}
Let $\{n_k\}$ an increasing sequence of positive integers, then for any $n \in\mathbb{N}$, there exist an integer $r(0\le r<n)$, a subsequence $\{n_{k_j}\}$ of $\{n_k\}$ and an increasing sequence of positive integers $\{q_j\}$ such that $n_{k_j}=nq_j+r$.
\end{lemma}

\begin{lemma}(\cite{q18})\label{yinlicd}
Suppose $X$ is a second countable space with the Baire property. If the non-autonomous discrete system $(X,f_{1,\infty})$ is topologically transitive, then it is point transitive.
\end{lemma}

\begin{lemma}(\cite{q14})\label{yinlijmaa1}
Let $(X,d)$ be a metric space without isolated points. Suppose that $f_n$ converges pointwise to $f$. Then

(1) If $p$ is a periodic point for $(X, f_{ 1,\infty} )$ then $p$ is a periodic point for $(X, f)$.

(2) If there exists an infinite set of periodic points of $f$, then there exists $\eta>0$ such that for any $x\in X$ there is a periodic point $p$ of $f$ such that $d(x,f^n(p))\geq\eta$ for all $n \in\mathbb{N}$.

\end{lemma}

\begin{lemma}(\cite{q14})\label{yinlijmaa2}
Let $(X,d)$ be a metric space without isolated points. If $(X, f_{ 1,\infty} )$ is transitive, then for any pair of non-empty open subsets $U,V$ of $X$, the set $N_{f_{ 1,\infty}}(U,V)$ is infinite.
\end{lemma}

\section{Main results}

In this section, we focus on sensitivity, transitivity and chaos for NDS and their relationships. In the first two theorems, we give some different sufficient conditions for NDS to be chaotic. Then, we obtain several sufficient conditions for NDS to be sensitive. Moreover, we provide some examples to show a significant difference between the theory of ADS and the theory of NDS. Finally, we introduce the concept of  weakly sensitivity for non-autonomous discrete systems.

\begin{theorem}\label{budongdian0}
Let $(X,f_{1,\infty})$ be a compact and  commutative non-autonomous discrete system without isolated points. If $(X,f_{1,\infty})$ is topologically transitive and has a fixed point, then $(X,f_{1,\infty})$ is Li-Yorke chaos.
\end{theorem}

To prove this theorem, we begin by giving some lemmas.
\begin{lemma}(\cite{q1})\label{budongdian1}
Let $h_i:X\rightarrow [0,+\infty]$ be upper semi-continuous for any $i\in\mathbb{N}$ and $a\in [0,+\infty]$. Define $$g(x)=\liminf_{i\rightarrow\infty}h_i(x), x\in X.$$ If $\overline{g^{-1}([0,a])}=X$, then $g^{-1}([0,a])$ is a $G_\delta$ set.
\end{lemma}

\begin{lemma}\label{budongdian2}
Let $(X,f_{1,\infty})$ be a commutative non-autonomous discrete system  without isolated points. If $(X,f_{1,\infty})$ is point transitive and has fixed point, then there exists a dense $G_\delta$ set $B$ of $X\times X$ such that for any $(x_1,x_2)\in B$, $$\liminf_{n\rightarrow\infty}d(f_1^n(x_1),f_1^n(x_2))=0.$$
\end{lemma}
\begin{proof}
Define $F:X\times X\rightarrow R$ by
$$F(x_1,x_2)=\liminf_{n\rightarrow\infty}d(f_1^n(x_1),f_1^n(x_2)), \forall (x_1,x_2)\in X\times X.$$
Let $\omega$ be a transitive point of $(X,f_{1,\infty})$. Then for any $z=(z_1,z_2)\in orb(\omega,f_{1,\infty})\times orb(\omega,f_{1,\infty})$, there exist $k_1,k_2\in\mathbb{Z_+}$ such that $f_1^{k_1}(\omega)=z_1$ and $f_1^{k_2}(\omega)=z_2$. Let $v$ be the fixed point of $(X,f_{1,\infty})$. Then there exists an increasing sequence $\{n_j\}$ of positive integers such that $$\lim_{j\rightarrow\infty}f_1^{n_j}(\omega)=v.$$ Besides, since $(X,f_{1,\infty})$ is commutative, therefore, for each $i=1,2$, $$\lim_{j\rightarrow\infty}f_1^{n_j}(z_i)=\lim_{j\rightarrow\infty}f_1^{n_j}(f_1^{k_i}(\omega))
=\lim_{j\rightarrow\infty}f_1^{k_i}(f_1^{n_j}(\omega))=f_1^{k_i}(v)=v.$$
This shows that $F(z)=0$. Since $\overline{orb(\omega,f_{1,\infty})\times orb(\omega,f_{1,\infty})}=X\times X$, $F^{-1}(0)$ is a dense $G_\delta$ set of $X\times X$ by lemma \ref{budongdian1}.

\end{proof}

\begin{lemma}\label{budongdian3}
If $R(f_{1,\infty})$ is dense in $X$, then $R(f_{1,\infty})$ is a dense $G_\delta$ set.
\end{lemma}

\begin{proof}
Define $F:X\rightarrow [0,+\infty]$ by
$$F(x)=\liminf_{n\rightarrow\infty}d(f_1^n(x),x), \forall x\in X.$$
It is obvious that $x\in R(f_{1,\infty})\Leftrightarrow F(x)=0$. Thus, $F^{-1}(0)$ is dense in $X$. By lemma \ref{budongdian1}, $R(f_{1,\infty})=F^{-1}(0)$ is a dense $G_\delta$ set of $X$.

\end{proof}

\begin{lemma}\label{budongdian4}
Let $(X,f_{1,\infty})$ be a compact non-autonomous discrete system. If $(X,f_{1,\infty})$ is topologically transitive, then $R(f_{1,\infty}\times f_{1,\infty})$ is a dense $G_\delta$ set of $X\times X$.
\end{lemma}

\begin{proof}
Since $(X,f_{1,\infty})$ is topologically transitive, there exists a transitive point $\omega\in X$ by lemma \ref{yinlicd}, then $orb(\omega,f_{1,\infty})\times orb(\omega,f_{1,\infty})$ is dense in $X\times X$ and $orb(\omega,f_{1,\infty})\times orb(\omega,f_{1,\infty})\subseteq R(f_{1,\infty}\times f_{1,\infty})$. Thus, $R(f_{1,\infty}\times f_{1,\infty})$ is dense in $X\times X$. By lemma \ref{budongdian3}, $R(f_{1,\infty}\times f_{1,\infty})$ is a dense $G_\delta$ set of $X\times X$.

\end{proof}

\begin{lemma}(\cite{q1})\label{budongdian5}
Suppose that $X$ is a separable complete metric space dense in itself. If for any $n\in \mathbb{N}$, the set $P_n$ is residual in the product space $X^n$. Then there is a Mycielski set $K$ in $X$ such that for any $n\in \mathbb{N}$ and any pairwise different $n$ points $x_1,x_2,...,x_n$ in $K$, we have $(x_1,x_2,...,x_n)\in P_n$.
\end{lemma}

\noindent {\textit{Proof of Theorem \ref{budongdian0}}} : Put $D=R(f_{1,\infty}\times f_{1,\infty})\cap B$, where $B$ is identical in lemma \ref{budongdian2}. By lemma \ref{budongdian2} and lemma \ref{budongdian4}, $D\subseteq X\times X$ and it is a residual set. Since $X$ is compact, it is both complete and separable. By lemma \ref{budongdian5}, there exists a dense Mycielski set $K$ in $X$ such that for any $x_1,x_2\in K$ with $x_1\neq x_2$, $(x_1,x_2)\in D$. Therefore, $(X,f_{1,\infty})$ is Li-Yorke chaos.


The following result improves Theorem 3.4 in \cite{q28}.
\begin{theorem}\label{kato}
Let $(X,f_{1,\infty})$ be a non-autonomous discrete system. If
$f_{1,\infty}$ is weakly mixing, then $f_{1,\infty}$ is Kato's chaos.
\end{theorem}

\begin{proof}
By definition of Kato's chaos, it suffices to show that $f_{1,\infty}$ is  accessible  and sensitive.

(i) For any $\varepsilon>0$ and nonempty open sets $U,V,W$ of $X$ with $diam(W)<\varepsilon/2$.
Since $f_{1,\infty}$ is weakly mixing, there exists $n\in \mathbb{N}$ such that $f_1^n(U)\cap W \neq \emptyset$ and $f_1^n(V)\cap W \neq \emptyset$.
That is, $\exists x\in U,y\in V$ satisfying $f_1^n(x)\in W$ and $f_1^n(y)\in W$. Therefore $d(f_1^n(x),f_1^n(y))\leq diam(W)<\varepsilon$. This means that $f_{1,\infty}$ is  accessible.

(ii) Take $x_1,x_2\in X$ with $x_1\neq x_2$, Note $r=d(x_1,x_2)$ and $\delta:=r/2$.  For any nonempty open set $U$ of $X$. Since $f_{1,\infty}$ is weakly mixing, there exists $m\in\mathbb{N}$ such that $f_1^m(U)\cap B(x_1,r/4) \neq \emptyset$ and $f_1^n(U)\cap B(x_2,r/4) \neq \emptyset$. That is, $\exists y_1\in U$ such that $f_1^n(y_1)\in B(x_1,r/4)$ and  $\exists y_2\in U$ such that $f_1^n(y_1)\in B(x_2,r/4)$. So we have $d(f_1^n(y_1),x_1)\leq r/4$ and $d(f_1^n(y_2),x_2)\leq r/4$. Therefore,
$d(f_1^n(y_1),f_1^n(y_2))> r/2=\delta$, which follows that $f_{1,\infty}$ is sensitive.
\end{proof}

It is well known that in the three conditions defining Denvaney chaos in ADS, topological transitivity and dense periodic points together imply sensitivity. However,  the result no longer holds for non-autonomous discrete system in general \cite{q14}. In \cite{q14}, the authors proved that if the sequence $f_n$ converges uniformly to $f$, then sensitivity follows. And in \cite{q2}, the authors proved that if the sequence $f_n$ is finitely generated and existing two invariant periodic points, then sensitivity follows. The following two theorems extend the results in \cite{q14} and \cite{q2}.

\begin{theorem}
Let $(X,f_{1,\infty})$ be a finitely generated non-autonomous discrete system. If

(1)$(X,f_{1,\infty})$ is topologically transitive,

(2)periodic points are dense in $X$,

(3)existing two invariant periodic points $x,y\in X$ with $\overline{orb(x,f_{1,\infty})}\cap \overline{orb(y,f_{1,\infty})}= \emptyset$,

then $(X,f_{1,\infty})$ is sensitive.
\end{theorem}

\begin{proof}
Let $F=\{f_1,f_2,...,f_m\}$. For any nonempty subset $A$ of $X$ and $x\in X$, write $d(x,A)=inf\{d(x,y)\mid y\in A\}$. Let the two invariant periodic points are $p_1$ and $p_2$ with $\overline{orb(p_1,f_{1,\infty})}\cap \overline{orb(p_2,f_{1,\infty})}= \emptyset$. Then we can note
$$\delta=\frac{1}{4}inf\{d(x,y)\mid x\in orb(p_1,f_{1,\infty}), y\in orb(p_2,f_{1,\infty})\}>0.$$
Firstly, we will claim that for any $x\in X$, there exists an invariant periodic point $p_x\in \{p_1,p_2\}$ satisfying $d(x,orb(p_x,f_{1,\infty}))>\delta$.
To this end, consider the following three cases:

(1)if $x\in orb(p_1,f_{1,\infty})$, then $d(x,orb(p_2,f_{1,\infty}))\geq4\delta>\delta$.

(2)if $x\in orb(p_2,f_{1,\infty})$, then $d(x,orb(p_1,f_{1,\infty}))\geq4\delta>\delta$.

(3)if $x\notin orb(p_1,f_{1,\infty})\cup orb(p_2,f_{1,\infty})$, since $d(orb(p_1,f_{1,\infty}),orb(p_2,f_{1,\infty}))\geq 4\delta$, then using the triangle inequality, we have $d(x,orb(p_1,f_{1,\infty}))>\delta$ or $d(x,orb(p_2,f_{1,\infty}))>\delta$.

Nextly, We will see that $(X,f_{1,\infty})$ is sensitive with sensitivity constant $\gamma:=\frac{\delta}{4}$. For any non-empty open subsets $U$ of $X$, there exist $x\in U$ and $\varepsilon>0$ such that $B(x,\varepsilon)\subseteq U$.
For the $x$, there is a invariant periodic point $p$ such that for all $i,n\in\mathbb{N}$

\begin{equation}\label{shi3.02}
d(x,f_i^n(p))\geq\delta
\end{equation}

from the above discussion. Since periodic points are dense in $X$, there is a periodic point $q$ such that
$$d(x,q)<min\{\varepsilon,\gamma\}.$$
Let $N$ be period of $q$, that is $f_1^{jN}(q)=q$ for any $j\in\mathbb{N}$. Since the function $f_i$ is continuous for any $i=1,2,...,m$, there exists neighborhood $V$ of $p$ such that for any $k=0,1,2,...,N$, $i\in\mathbb{N}$ and $y\in V$,
\begin{equation}\label{shi3.01}
d(f_i^k(p),f_i^k(y))<\gamma.
\end{equation}

Since $(X,f_{1,\infty})$ is topologically transitive, there exists a positive integer $k$ such that $f_1^{k}(B(x,\varepsilon))\cap V\neq \emptyset$, that is existing $z\in B(x,\varepsilon)$ such that $f_1^{k}(z)\in V$.

Take a positive integer $j$ satisfying $Nj=k+l$, where $0\leq l<N$. Denote $F:=f_{jN}\circ f_{jN-1}\circ\cdot\cdot\cdot\circ f_{(j-1)N+l+1}$.

We obtain using inequality (\ref{shi3.01}) that

$$d(F(p),F(f_1^{k}(z)))\le \gamma.$$

This and (\ref{shi3.02}) yield that

\begin{equation}
\begin{aligned}
&d(f_1^{jN}(q),f_1^{jN}(z))\\
=&d(q,F(f_1^{k}(z))) \\
\ge & d(x,F(p))-d(F(p),F(f_1^{k}(z)))-d(x,q)\\
>& 4\gamma-\gamma-\gamma\\
=&2\gamma.
\end{aligned}
\end{equation}

Therefore, either $$d(f_1^{jN}(x),f_1^{jN}(q))>\gamma$$ or $$d(f_1^{jN}(x),f_1^{jN}(z))>\gamma.$$
The proof of theorem holds.

\end{proof}

\begin{remark}
The conditions of theorem above are weaker than Theorem 3.1 in \cite{q2}. In \cite{q2}, a point $x\in X$ is a periodic point if there is a positive integer $N$ such that for any natural number $k$, $f_1^k(x)=f_1^{N+k}(x)$. This definition of periodic point is stronger than the one in this paper.  Besides, in Theorem 3.1 \cite{q2} the condition 'without isolate points' can be dropped. Additionally, the proof of theorem above is simpler than the proof of Theorem 3.1 in \cite{q2}.
\end{remark}

\begin{theorem}
Let $(X,f_{1,\infty})$ be a non-autonomous discrete system without isolated points. Assume that $f_{1,\infty}$ converges uniformly to $f$. If

(1)$(X,f_{1,\infty})$ is topologically transitive,

(2)periodic points are dense in $X$,

then $(X,f_{1,\infty})$ is multi-sensitive.
\end{theorem}

\begin{proof}
 For any $m\in \mathbb{N}$ and any non-empty open subsets $U_1,U_2,...,U_m$ of $X$, there exist $x_i\in U_i$ and $\varepsilon$ such that $B(x_i,\varepsilon)\subseteq U_i$ $(i=1,2,...,m)$.
For the above $x_i$, by lemma \ref{yinlijmaa1}, there is a periodic point $p_i$ such that for all $n\in\mathbb{N}$,
\begin{equation}\label{shi3.0}
d(x_i,f^n(p_i))\geq\eta.
\end{equation}
Next we will claim that $(X,f_{1,\infty})$ is multi-sensitive with sensitivity constant $\delta:=\frac{\eta}{4}$.
Since periodic points are dense in $X$, there is a periodic point $q_i$ such that
$$d(x_i,q_i)<min\{\varepsilon,\frac{\eta}{4}\}.$$
Let $N_i$ be period of $q_i$, that is, $f_1^{jN_i}(q_i)=q_i$ for any $j\in\mathbb{N}$. Since the functions $f^j$ are continuous for all $j=1,2,...,N_i$, there exists some neighborhood $V_i$ of $p_i$ such that for any $j=0,1,2,...,N_i$ and any $y\in V_i$,
\begin{equation}\label{shi3.1}
d(f^j(p),f^j(y))<\frac{\delta}{2}.
\end{equation}
Since $(X,f_{1,\infty})$ converges uniformly to $f$, there exists $j_0\in\mathbb{N}$ such that for any $j\geq j_0, 0\leq l_i<N_i,v\in V_i$, we have
$$d(f_{jN_i}\circ f_{jN_i-1}\circ\cdot\cdot\cdot\circ f_{(j-1)N_i+l_i+1}(v),f^{N_i-l}(v))<\frac{\delta}{2}.$$
Since $(X,f_{1,\infty})$ is topologically transitive and  has no isolated point, there exists an infinite set of positive integers $m_i$ such that $f_1^{m_i}(B(x_i,\varepsilon))\cap V_i\neq \emptyset$ by lemma \ref{yinlijmaa2}. Hence we take some $k_i$ satisfying $k_i=(j-1)N_i+l_i$ and $f_1^{k_i}(B(x_i,\varepsilon))\cap V_i\neq \emptyset$, where $0\leq l_i<N_i$ and $j>j_0$. Then take $z_i\in B(x_i,\varepsilon)$ such that $f_1^{k_i}(z_i)\in V$.

Denote $F_i:=f_{jN_i}\circ f_{jN_i-1}\circ\cdot\cdot\cdot\circ f_{(j-1)N_i+l_i+1}$. By hypothesis we can get that
$$d(F_i(f_1^{k_i}(z_i)),f^{N_i-l_i}(f_1^{k_i}(z_i)))<\frac{\delta}{2}.$$

By the triangle inequality we obtain using inequality (\ref{shi3.1}) that

\begin{equation}\label{shi3.2}
\begin{aligned}
&d(f^{N_i-l_i}(p),F_i(f_1^{k_i}(z_i)))\\
\le &d(F_i(f_1^{k_i}(z_i)),f^{N_i-l_i}(f_1^{k_i}(z_i)))
+d(f^{N-i}(p),f^{N_i-l_i}(f_1^{k_i}(z_i)))\\
\le & \frac{\delta}{2}+\frac{\delta}{2}\\
=&\delta.
\end{aligned}
\end{equation}
(\ref{shi3.2}) and (\ref{shi3.0}) yield that

\begin{equation*}
\begin{aligned}
&d(f_1^{jN_i}(q_i),f_1^{jN_i}(z_i))\\
=&d(q_i,F_i(f_1^{k_i}(z_i))) \\
\ge & d(x_i,f^{N_i-l_i}(p))-d(f^{N_i-l_i}(p),F_i(f_1^{k_i}(z_i)))-d(x_i,q_i)\\
>& 4\delta-\delta-\delta\\
=&2\delta.
\end{aligned}
\end{equation*}

Therefore, either $$d(f_1^{jN_i}(x_i),f_1^{jN_i}(q_i))>\delta$$ or $$d(f_1^{jN_i}(x_i),f_1^{jN_i}(z_i))>\delta.$$
 Take $N=N_1N_2...N_m$. Then for any $i=1,2,...,m$, either $$d(f_1^{jN}(x_i),f_1^{jN}(q_i))>\delta$$ or $$d(f_1^{jN_i}(x_i),f_1^{jN_i}(z_i))>\delta.$$ The proof of theorem holds.

\end{proof}

\begin{remark}
With the similar argument, we can also obtain that $(X,f_{1,\infty})$ is syndetically sensitive under the same assumptions above.
\end{remark}

\begin{theorem}\label{dengdumingan}
Let $(X,f_{1,\infty})$ be a compact and finitely generated non-autonomous discrete system. If
$f_{1,\infty}$ is transitive and commutative, then $f_{1,\infty}$ is either sensitive or almost equicontinuous.

\end{theorem}

\begin{proof}
Denote $F=\{f_1,f_2,...,f_m\}$ and
$$G_k=\{x\in X \mid \exists\delta>0 \text{  such that  }\forall n\in \mathbb{Z_+},\forall y,z\in B(x,\delta), d(f_1^n(y),f_1^n(z))<\frac{1}{k}\}.$$ The whole proof is divided into four steps.

Step 1. We will claim that $G_k$ is inversely invariant. Let $x\in f_i^{-1}(G_k)$, where $i=1,2,...,m$, then $f_i(x)\in G_k$. Let $\delta<\frac{1}{2k}$ satisfies the definition of $G_k$. Since $f_i$ is continuous and $X$ is compact, there exists $\eta<\delta$ such that for any $u,v\in X$ with $d(u,v)<\eta$, $d(f_i(u),f_i(v))<\delta$, where $i=1,2,...,m$. So for every $u,v\in B(x,\eta)$, we have $f_i(u),f_i(v)\in B(f_i(x),\delta)$. Therefore,
\begin{equation}\label{dengdu}
d(f_1^n\circ f_i(u),f_1^n\circ f_i(v))<\frac{1}{k}, \forall n\in \mathbb{Z_+},i=1,2,...,m.
\end{equation}
 Since $f_{1,\infty}$ is commutative, (\ref{dengdu}) implies that
$$d(f_i\circ f_1^n(u),f_i\circ f_1^n(v))<\frac{1}{k}, \forall n\in \mathbb{Z_+},i=1,2,...,m.$$
Further,
\begin{equation}\label{dengdu1}
d(f_1^{n+1}(u),f_1^{n+1}(v))<\frac{1}{k}, \forall n\in \mathbb{Z_+}.
\end{equation}
In addition, $d(u,v)<2\delta<\frac{1}{k}$. This together with (\ref{dengdu1}) can be verified that
$$d(f_1^{n}(u),f_1^{n}(v))<\frac{1}{k}, \forall n\in \mathbb{Z_+}.$$
Therefore, $x\in G_k$. Hence, $G_k$ is inversely invariant.

Step 2. We will see that $G_k$ is open.
Let $x\in G_k$ and let $\delta<\frac{1}{2k}$ satisfies the condition in the definition of $G_k$ for $x$. Then $B(x, \frac{\delta}{2})\subseteq G_k$. Indeed, for any $y\in B(x, \frac{\delta}{2})$ and any $n\in \mathbb{Z_+}$, if $z,w\in B(y,\frac{\delta}{2})$, then $z,w\in B(x, \delta)$. Therefore, $d(f_1^{n}(z),f_1^{n}(w))<\frac{1}{k}$. That is $B(x, \frac{\delta}{2})\subseteq G_k$. This shows that $G_k$ is open.

Step 3. We will prove that $Eq(f_{1,\infty})=\cap_{k=1}^\infty G_k$.
Let $x\in Eq(f_{1,\infty})$ and let $k\geq1$. Then there exists $\delta$ such that  $$d(x,y)<\delta \Rightarrow d(f_1^{n}(x),f_1^{n}(y))<\frac{1}{2k}, \forall n\in \mathbb{Z_+}.$$
So for any $y,z\in B(x, \delta)$, $d(f_1^{n}(y),f_1^{n}(z))<\frac{1}{k}, \forall n\in \mathbb{Z_+}$ by triangle inequality. Thus, $x\in G_k, \forall k\in \mathbb{Z_+}$. Therefore $Eq(f_{1,\infty})\subseteq \cap_{k=1}^\infty G_k$. Conversely, let $x\in \cap_{k=1}^\infty G_k$. For any $\varepsilon>0$, there exists $k\geq1$ such that $\frac{1}{k}<\varepsilon$. Since $x\in \cap_{k=1}^\infty G_k$,  there exists $\delta$ such that $\forall n\in \mathbb{Z_+},\forall y,z\in B(x,\delta), d(f_1^n(y),f_1^n(z))<\frac{1}{k}<\varepsilon$. Take $z=x$, then $d(f_1^n(y),f_1^n(x))<\varepsilon$. Hence $x\in Eq(f_{1,\infty})$, and further $\cap_{k=1}^\infty G_k\subseteq Eq(f_{1,\infty})$. Thus, $Eq(f_{1,\infty})=\cap_{k=1}^\infty G_k$.

Step 4. Let $(X,f_{1,\infty})$ be a transitive non-autonomous discrete system and $k\in \mathbb{N}$. Assume that $G_k$ is nonempty and nondense. Then $U:=X\backslash {\overline{G_k}}$ is open and nonempty, so there exist $n\in \mathbb{N}$ such that $U\cap f_1^{-n}(G_k)\neq \emptyset$. But $U\cap f_1^{-n}(G_k)\subseteq U\cap G_k=\emptyset$. This is a contradiction. Hence,  $G_k$ is either empty or dense. If all $G_k$ are nonempty, then they are dense. Thus, $Eq(f_{1,\infty})=\cap_{k=1}^\infty G_k$ is a residual set, that is $\overline{Eq(f_{1,\infty})}=X$, so $f_{1,\infty}$ is almost equicontinuous. If $G_k=\emptyset$ for some $k\in \mathbb{N}$, then for any $x\in X$ and any $\delta>0$, there exist $y,z\in B(x, \delta)$ and $n\in \mathbb{Z_+}$ such that $d(f_1^n(y),f_1^n(z))\geq\frac{1}{k}$. It follows that either $d(f_1^n(x),f_1^n(y))>\frac{1}{3k}$ or $d(f_1^n(x),f_1^n(z))>\frac{1}{3k}$. This implies that  $f_{1,\infty}$ is sensitive.

The entire proof is complete.
\end{proof}

\begin{corollary}
Let $(X,f_{1,\infty})$ be a finitely generated non-autonomous discrete system. If
$f_{1,\infty}$ is minimal(M2) and commutative, then $f_{1,\infty}$ is either sensitive or equicontinuous.
\end{corollary}

\begin{proof}
Since $f_{1,\infty}$ is minimal, then for any $x\in X$, $\overline{orb(x)}=X$. If $f_{1,\infty}$ is nonsensitive, then $Eq(f_{1,\infty})\neq \emptyset$. Therefore $G_k\neq \emptyset,\forall k\in \mathbb{N}$ by the step 3 in the Theorem \ref{dengdumingan} above. Hence, for any $k\in \mathbb{N}$ and any $n\in \mathbb{Z_+}$, $orb(x,f_{1,\infty})\cap f_1^{-n}(G_k) \neq \emptyset$. So, there exists $p_k\in \mathbb{Z_+}$ such that $f_1^{p_k}(x)\in f_1^{-n}(G_k)$, that is, $f_1^{n}(x)\in f_1^{-p_k}(G_k)\subseteq G_k$. Therefore, $orb(x,f_{1,\infty})\subseteq G_k, \forall k\in \mathbb{N}$. Again, by the step 3 in the Theorem \ref{dengdumingan} above, one can get that $orb(x,f_{1,\infty})\subseteq Eq(f_{1,\infty})$. Thus, $X=Eq(f_{1,\infty})$ by the fact that $X=\overline{orb(x,f_{1,\infty})}\subseteq \overline{Eq(f_{1,\infty})}\subseteq X$, and further, $f_{1,\infty}$ is equicontinuous.

\end{proof}

\begin{theorem}
Let $(X,f_{1,\infty})$ be a nontrivial non-autonomous discrete system. If
$f_{1,\infty}$ is mixing, then $f_{1,\infty}$ is cofinitely sensitive.
\end{theorem}

\begin{proof}
As $X$ is not reduced to a single point, there exists a $\delta>0$ such that for any nonempty open sets $U\in X$ with $diam(U)<\delta$, there exists a nonempty open sets $V\in X$ satisfying $d(U,V)>\delta$. Since $f_{1,\infty}$ is mixing, there is $N\in \mathbb{N}$ such that for any $n\geq N$, $f_1^n(U)\cap U \neq \emptyset$ and $f_1^n(U)\cap V \neq \emptyset$. That is, there exist $x,y\in U$ such that for any $n\geq N$, $d(f_1^n(x),f_1^n(y))>\delta$. Therefore, $f_{1,\infty}$ is cofinitely sensitive.
\end{proof}

\begin{theorem}
Let $(X,f_{1,\infty})$ be a compact non-autonomous discrete system. Then the following conditions are equivalent:

(1)$f_{1,\infty}$ is minimal(M2),

(2)for any nonempty open set $U\subseteq X$, there exists a $n\in \mathbb{N}$ such that $X=\cap_{k=0}^n f_1^{-k}(U)$.
\end{theorem}

\begin{proof}

$(1)\Rightarrow(2)$. Since $f_{1,\infty}$ is minimal, then for any $x\in X$, $\overline{orb(x)}=X$. Therefore, for any nonempty open set $U\subseteq X$, there exists a $k\in \mathbb{Z_+}$ such that $f_1^k(x)\in U$, that is, $x\in f_1^{-k}(U)$. Hence, $X=\cap_{k=0}^\infty f_1^{-k}(U)$. Since $X$ is compact, there a $n\in \mathbb{N}$ such that $X=\cap_{k=0}^n f_1^{-k}(U)$.

$(2)\Rightarrow(1)$. Let $x\in X$ and $U\subseteq X$ be a nonempty open set. Since $X=\cap_{k=0}^n f_1^{-k}(U)$, where $n\in \mathbb{N}$, there exists a $r\in \mathbb{N}$ such that $f_1^r(x)\in U$, which implies that $\overline{orb(x)}=X$.

\end{proof}

For autonomous system, it is known that if $f$ is transitive, then it is surjective  \cite{q27}. But the result does not hold for non-autonomous system.

\begin{proposition}
There exists a non-autonomous discrete system $(X,f_{1,\infty})$ which is transitive but there are infinitely many $f_i\in f_{1,\infty}$ which are not surjective.
\end{proposition}
\begin{proof}
Let $X=[0,1]$, $g_1(x)=\frac{x}{2}, x\in [0,1]$,
\[g_2(x)=\begin{cases}
2x, &  x\in [0,\frac{1}{2}], \\
1, &  x\in [\frac{1}{2},1],
\end{cases}\]
and $g_3$ be a tent map, that is,
\[g_3(x)=\begin{cases}
2x, &  x\in [0,\frac{1}{2}], \\
2x-1, &  x\in [\frac{1}{2},1].
\end{cases}\]
The sequence $f_{1,\infty}$ is given by
\[f_i=\begin{cases}
g_1, &  i=3k+1, \\
g_2, &  i=3k+2,  \\
g_3, &  i=3k+3, 
\end{cases}\]
where $ k\in \mathbb{Z_+}$, that is $f_{1,\infty}=\{g_1,g_2,...g_n,...\}$.
It is easy to see that all $f_{3k+1}=g_1$ are not surjective where $k\in \mathbb{Z_+}$.
\end{proof}

For autonomous system, it is known that  if $f$ is transitive, then $N_f(U,V)$ is infinite, and that if transitive system $(X,f)$ has isolated point, then $X$ is a finite periodic orbit \cite{q27,q7}. However, the example below shows that these conclusions above are unavailable for non-autonomous discrete system in general.

\begin{example}

Let $(X,d)$ be a metric space without isolated points, $a$ be a point with $d(a,X)>0$, and $f:X\rightarrow X$ be transitive and continuous. Since $(X,f)$ is transitive, there exists a transitive point $x_0\in X$. Let $Y:=X\cup \{a\}$, $f_1(x)\equiv a$, $f_2(x)\equiv x_0$, $f_i(x)\equiv f$ $(i=3,4,...)$.

For any nonempty open set $U\subseteq Y$ and the nonempty open set $V=\{a\}$, it is clear that $N_f(U,V)$ is finite. Moreover, the space $Y$ contains a isolated point but it is not a finite periodic orbit.
\end{example}

An interesting aspect of the theory of discrete dynamical systems lies in the fact that if a point $x_0$ is almost periodic if and only if it is minimal. As the two examples below shows, the situation is quite different if we replace ADS by NDS. That is to see, they are mutually nonequivalent for NDS.

\begin{example}

Let $f:S^1\rightarrow S^1$ be the map defined by $f(e^{2\pi i\theta})=e^{2\pi i(\theta+\alpha)}$, where $S^1$ is a unit circle on the complex plane, $\theta\in [0,1]$ and $\alpha$ is a irrational numbers. Let
$$f_{1,\infty}=\{f_1,f_2,...f_n,...\}=\{f, f^{-1}, f, f^{-1},f, f^{-1},... \}.$$

It clear that each $x\in S^1$ is periodic since $f_1^{2k}=id$ for any $k\in \mathbb{N}$. Further, each $x\in S^1$ is almost periodic. However, each $x\in S^1$ is not minimal since the orbit $orb(x,f_{1,\infty})=\{e^{2\pi i\theta},e^{2\pi i(\theta+\alpha)}\}$ is not an invariant set of $f_{1,\infty}$.

\end{example}

\begin{example}

Let $X=\{1,2,3\}$ and $f:X\rightarrow X$ be the map defined by $f(1)=2$,$f(2)=3$,$f(3)=1$. Let
\begin{equation*}
\begin{aligned}
f_{1,\infty}=&\{f_1,f_2,...f_n,...\}\\
=&\{f,id,f,id,id,f,id,id,id,...,f,\underbrace{id,id,...id}_{n-fold},... \}.
\end{aligned}
\end{equation*}

It is is easy to verify that each $x\in X$ is minimal. However, each $x\in S^1$ is not almost periodic.

\end{example}

In autonomous system it is known that both the definitions (M1) and (M2) are equivalent. However, the result is not true in non-autonomous system in general. The following theorem and example justify it.

\begin{theorem}\label{minimal1}
(M2) implies (M1).
\end{theorem}
\begin{proof}
Let $X$ be (M2). Assume if possible that $X$ is not (M1), then there exists a minimal set $E\subseteq X$ with $E\neq X$ such that $f_i(E)\subseteq E$ for any $i\in\mathbb{Z_+}$. Take $x\in E$, then $\overline{orb(x)}\subseteq E$. Since $X$ is (M2), then $\overline{orb(x,f_{1,\infty})}=X$, therefore $X\subseteq E$, which is a contradiction and hence $X$ is (M1).

\end{proof}

\begin{example}\label{minimal2}
Let $X=[0,1]$.
Let $f_{5k+1}=id$  where $k\in \mathbb{Z^+}$, $f_{5k+2}=\frac{x}{2}$, 
\[f_{5k+3}=\begin{cases}
2x, &  x\in [0,\frac{1}{2}], \\
1, &  x\in [\frac{1}{2},1],
\end{cases}\]

\[f_{5k+4}=\begin{cases}
x+\frac{1}{2^{k+2}}, &  x\in [0,\frac{1}{2}], \\
(1-\frac{1}{2^{k+1}})x+\frac{1}{2^{k+1}}, &  x\in [\frac{1}{2},1],
\end{cases}\]

\[f_{5k+5}=\begin{cases}
0, &  x\in [0,\frac{1}{2^{k+2}}], \\
x-\frac{1}{2^{k+2}}, &  x\in [\frac{1}{2^{k+2}},\frac{1}{2}+\frac{1}{2^{k+2}}], \\
\frac{2^{k+1}}{2^{k+1}-1}x-\frac{1}{2^{k+1}-1}, &  x\in [\frac{1}{2}+\frac{1}{2^{k+2}},1],
\end{cases}\]
where $k\in \mathbb{Z_+}$.

Note that $f_{5k+1}=id$, $f_{5k+3}\circ f_{5k+2}=id$ and $f_{5k+5}\circ f_{5k+4}=id$ for all $k\in \mathbb{Z_+}$.
For the non-autonomous discrete system $(X,f_{1,\infty}=\{f_1,f_2,...\})$, it is easy to see that $orb(1,{f_{1,\infty}})=\{\frac{1}{2},1\}$, therefore $\overline{orb(1,{f_{1,\infty}})}\neq [0,1]$, further $X$ is not (M2). Next we will claim that $X$ is (M1). If not, there exists a non-empty closed and invariant subset $M\subset [0,1]$ with $M\neq X$. Further, there exists a closed interval $[a,b]\subseteq M$ such that there is no closed intervals $[c,d]$ satisfying $[a,b]\subseteq[c,d]\subseteq M$ with $[a,b]\neq[c,d]$ ($[a,b]$ degenerates to a point if $a=b$). We will prove that $a=0, b>\frac{1}{2}$. For this end, firstly, assume that $a\neq0$, then there exists some $\varepsilon>0$ such that $[a-\varepsilon,a]\nsubseteq M$, further, there exists some $k_0\in \mathbb{Z_+}$ satisfying $\frac{1}{2^{k_0+2}}<\varepsilon$, therefore $f_{5k_0+5}(a)\notin M$, witch is a contradiction. Hence $a=0$. Secondly, assume that $b\leq \frac{1}{2}$, then $f_4([0,b])\nsubseteq M$, which is a contradiction. Therefore, $b>\frac{1}{2}$. For $[0,b]$ with $b>\frac{1}{2}$, we have $f_3([0,b])=[0,1]\nsubseteq M$, which contradicts to $f_i([0,b])\subseteq M$. Therefore $(X,f_{1,\infty})$ is (M1).

\end{example}

The following example shows that there exists a a finitely generated non-minimal non-autonomous system on a compact metric space which is topologically transitive with dense minimal points but not thickly syndetically sensitive. This example disprove Theorem 5.1 proved by Salman and Das \cite{q5}. We modify one of the conditions in {\rm\cite[Theorem 5.1]{q5}} and reprove it in the following Theorem.
\begin{example}
Consider the shift map $\sigma:\Sigma_2\rightarrow\Sigma_2$ defined in {\rm\cite[Example 5.1]{q5}}.

Let
\begin{equation*}
\begin{aligned}
f_{1,\infty}=&\{f_1,f_2,...f_n,...\}\\
=&\{\sigma, id,  \sigma^{-1}, id, \sigma, id, id, \sigma, id, id, \sigma^{-1}, id, id,\sigma^{-1}, id, id ,..., \\
& \overbrace{\sigma, \underbrace{id,...,id}_{n-fold}, \sigma, id,...,id,..., \sigma, id,...,id}^{n-fold},\\
&\overbrace{\sigma^{-1}, \underbrace{id,...,id}_{n-fold},\sigma^{-1}, id,...,id,..., \sigma^{-1}, id,...,id}^{n-fold},...\}.
\end{aligned}
\end{equation*}

In \cite{q5}, a non-empty subset $U\subseteq X$ is called an invariant set of $f_{ 1,\infty}$ if $f_1^n(U)\subseteq U$ for any $n\in \mathbb{N}$, a set $U$ is called minimal subset of $f_{ 1,\infty}$ if it is a non-empty, invariant closed subset of $X$ and no proper subset of $U$ is non-empty, invariant and closed, and a point $x\in X$ is called minimal point if its closure orbit is a minimal subset of $f_{ 1,\infty}$.

To verify the result we use a similar argument to the example 5.1 given in  \cite{q5}. Note that the $orb_{f_{1,\infty}}(x)=\{x,\sigma(x),\sigma^2(x),...\}=orb_\sigma(x)$. Since the minimal points of $(\Sigma_2,\sigma)$ are dense in $\Sigma_2$, the minimal points of $(\Sigma_2,f_{1,\infty})$ are also dense in $\Sigma_2$. Besides, $(\Sigma_2,\sigma)$ is non-minimal autonomous system on a compact metric space, so $(\Sigma_2,f_{1,\infty})$ is a finitely generated non-minimal non-autonomous system on a compact metric space. However, it is obvious that $(\Sigma_2,f_{1,\infty})$ cannot be syndetically sensitive, hence, $(\Sigma_2,f_{1,\infty})$ cannot be thickly syndetically sensitive.

\end{example}

\begin{theorem}
Let $(X,f_{1,\infty})$ be a finitely generated non-autonomous discrete system on a compact metric space. Assume that

(1)$f_{1,\infty}$ is nonminimal (in the sense of (M1)),

(2)$f_{1,\infty}$ is topologically transitive,

(3)$f_{1,\infty}$ has set of almost periodic points dense.

Then it is thickly syndetically sensitive.
\end{theorem}

\begin{proof}
Let  $U$ be a non-empty open subset of $X$ and $S, T$ be a pair of minimal subsets of $(X,f_{1,\infty})$ with $d(S,T):=a>0$. Since $(X,f_{1,\infty})$ is finitely generated and $X$ is compact, therefore, for the $a$ above and any $k\in \mathbb{N}$, there exists $\delta>0$ such that for all $x,y\in X$, for each $j=\{1,2,...,k\}$ and for any $n\geq 1$, we have

\begin{equation}\label{thickmg}
d(x,y)<\delta\Rightarrow d(f_n^j(x),f_n^j(y))<\frac{a}{4}.
\end{equation}

Besides, since $X$ is second countable and complete, we get that $(X,f_{1,\infty})$ is point transitive by lemma \ref{yinlicd}. For any transitive point $z\in U$ and minimal set $S$, there is $m\in \mathbb{N}$ such that $d(f_1^m(z),S)<\frac{\delta}{2}$. As $(X,f_{1,\infty})$ has dense set of almost periodic points, there exists an almost point $w\in U$ sufficiently close to $z$. In addition, $f_1^m$, for any fixed natural number $m$, is uniformly continuous, therefore, we have $d(f_1^m(z),f_1^m(w))<\frac{\delta}{2}$ and hence using triangle inequality one can get that $d(f_1^m(w),S)<\delta$. Since $w$ is an almost periodic point, there exists a syndetic set $\{m_j:j\in\mathbb{N}\}$ satisfying $d(f_1^{m_j}(w),S)<\delta$. Therefore, by (\ref{thickmg}), we have $d(f_{m_j+1}^i\circ f_1^{m_j}(w),f_{m_j+1}^i(S))<\frac{a}{4}$ and using the minimality of $S$, we get that $d(f_{m_j+1}^i\circ f_1^{m_j}(w),S)<\frac{a}{4}$ for every $i=1,2,...,k$ and $j\in \mathbb{N}$. Now, since $k$ is an arbitrary natural number, therefore $N(U,B(S,a/4))$ is thickly syndetic, where $N(x,U)=\{n\in \mathbb{N}: f_1^n(x)\in U\}$. By similar arguments, one can get that $N(U,B(T,a/4))$ is thickly syndetic. Hence, $N(U,B(S,a/4))\cap N(U,B(T,a/4))$ is syndetic. Thus, for every $r\in N(U,B(S,a/4))\cap N(U,B(T,a/4))$, we have
\begin{equation}\label{thickmg1}
\{r,r+1,r+2,...,r+k\}\subseteq N(U,B(S,a/4))\cap N(U,B(T,a/4)).
\end{equation}
Since $k$ is arbitrary, we have
 $$\{r\in \mathbb{N}: r+j\in N(U,B(S,a/4))\cap N(U,B(T,a/4)), \text{ for all } 1\leq j\leq k\}$$ is syndetic, for any $k\in \mathbb{N}$. Therefore, $N(U,B(S,a/4))\cap N(U,B(T,a/4))$ is thickly syndetic. Let $l\in N(U,B(S,a/4))\cap N(U,B(T,a/4))$ and hence there exist $u,v\in U$ such that $d(f_1^{l}(u),S)<a/4$ and $d(f_1^{l}(v),T)<a/4$. Consequently, by the triangle inequality $d(S,T)\leq d(f_1^{l}(u),S)+d(f_1^{l}(u),f_1^{l}(v))+d(f_1^{l}(v),T)$, we have
\begin{equation*}
\begin{aligned}
d(f_1^{l}(u),f_1^{l}(v))&\geq d(S,T)-d(f_1^{l}(u),S)-d(f_1^{l}(v),T) \\
&> a-\frac{a}{4}-\frac{a}{4} \\
&=\frac{a}{2}.
\end{aligned}
\end{equation*}
Hence, $l\in N_{f_{1,\infty}}(U,a/2)$, which implies that
$$N(U,B(S,a/4))\cap N(U,B(T,a/4))\subseteq N_{f_{1,\infty}}(U,a/2).$$
 Thus, by (\ref{thickmg1}), we have $\{r,r+1,r+2,...,r+k\}\subseteq N_{f_{1,\infty}}(U,a/2)$ and since $k$ is arbitrary, therefore we get that $N_{f_{1,\infty}}(U,a/2)$ is thickly syndetic for every open non-empty subset of $X$ and some $a>0$, which implies that $f_{1,\infty}$ is thickly syndetically sensitive.

\end{proof}

\begin{corollary}
Let $(X,f_{1,\infty})$ be a finitely generated non-autonomous discrete system on a compact metric space. Assume that $f_{1,\infty}$ is topologically transitive and has set of periodic points dense. Then it is thickly syndetically sensitive.
\end{corollary}

In \cite{q18,q5}, the authors studied that for $k$-periodic non-autonomous discrete system, the relationships of multi-sensitivity and property P between non-autonomous discrete system $(X,f_{1,\infty})$ and induced  autonomous system $(X,f_k\circ\cdots f_1)$. The following four theorems  explore that the relationships of other  dynamical properties between non-autonomous discrete system $(X,f_{1,\infty})$ and induced  autonomous system $(X,g=f_k\circ\cdots f_1)$.

\begin{theorem}
Let $(X,f_{1,\infty})$ be a $k$-periodic non-autonomous discrete system and $g=f_k\circ f_{k-1}\circ\cdots f_1$. If the autonomous discrete system $(X,g)$ is Li-Yorke sensitive, then $(X,f_{1,\infty})$ is also Li-Yorke sensitive. Converse is true when each $f_p(p=1,2,...,k)$ is uniformly continuous.
\end{theorem}

\begin{proof}

Suppose that $(X,g)$ is Li-Yorke sensitive, then there is a sensitive constant $\delta>0$ such that for any non-empty open subset $U\subseteq X$, there exist $x,y\in U$ satisfying
$$\liminf_{n\rightarrow\infty}d(g^n(x),g^n(y))=0,$$
and
$$\limsup_{n\rightarrow\infty}d(g^n(x),g^n(y))\geq\delta.$$
Thus there exist two increasing sequence $\{n_i\}$ and $\{m_i\}$ of positive integers such that
$$\lim_{i\rightarrow\infty}d(g^{n_i}(x),g^{n_i}(y))=0,$$
and
$$\lim_{i\rightarrow\infty}d(g^{m_i}(x),g^{m_i}(y))\geq\delta.$$
Since $g=f_k\circ f_{k-1}\circ\cdots f_1$, then we have $g^{n_i}=(f_1^k)^{n_i}=f_1^{kn_i}$ and $g^{m_i}=(f_1^k)^{m_i}=f_1^{km_i}$.
Therefore,
$$\lim_{i\rightarrow\infty}d(f_1^{kn_i}(x),f_1^{kn_i}(y))=0,$$
and
$$\lim_{i\rightarrow\infty}d(f_1^{km_i}(x),f_1^{km_i}(y))\geq\delta.$$
It follows that $(X,f_{1,\infty})$ is Li-Yorke sensitive.

Conversely, let $(X,f_{1,\infty})$ be Li-Yorke sensitive with $\delta>0$ as a constant of sensitivity. Then there exist two increasing sequence $\{n_j\}$ and $\{m_j\}$ of positive integers such that
\begin{equation}\label{lymg1}
\lim_{j\rightarrow\infty}d(f_1^{n_j}(x),f_1^{n_j}(y))=0,
\end{equation}
and
\begin{equation}\label{lymg2}
\lim_{j\rightarrow\infty}d(f_1^{m_j}(x),f_1^{m_j}(y))\geq\delta.
\end{equation}
By lemma \ref{yinli1},  there exist integers $r_1,r_2(0\le r_1,r_2<n)$, subsequence $\{n_{j_l}\}$ of $\{n_j\}$, $\{m_{j_l}\}$ of $\{m_j\}$, and increasing sequence of positive integers $\{q_{1j}\}$ and  $\{q_{2j}\}$ such that $n_{j_l}=kq_{1j}-r_1$ and $m_{j_l}=kq_{2j}+r_2$.
Since $g=f_k\circ f_{k-1}\circ\cdots f_1$, then (\ref{lymg1}) and (\ref{lymg2}) imply that for any $j\in \mathbb{N}$,
\begin{equation}\label{lymg3}
\lim_{j\rightarrow\infty}d(f_1^{n_{j_l}}(x),f_1^{n_{j_l}}(y))=0
\end{equation}
and
\begin{equation}\label{lymg4}
\begin{aligned}
&\lim_{j\rightarrow\infty}d(f_{kq_{2j}+1}^{r_2}\circ f_1^{kq_{2j}}(x),f_{kq_{2j}+1}^{r_2}\circ f_1^{kq_{2j}}(y))\\
=&\lim_{j\rightarrow\infty}d(f_{kq_{2j}+1}^{r_2}\circ g^{q_{2j}}(x),f_{kq_{2j}+1}^{r_2}\circ g^{q_{2j}}(y))\\
\geq&\delta.
\end{aligned}
\end{equation}
Since $(X,f_{1,\infty})$ is $k$-periodic and each $f_p(p=1,2,...,k)$ is uniformly continuous, then for any integer $0\leq r<k$, $f_p^r$ is uniformly continuous. Therefore, ({\ref{lymg3}}) implies that

\begin{equation}\label{lymg5}
\begin{aligned}
&\lim_{j\rightarrow\infty}d(g^{q_{1j}}(x),g^{q_{1j}}(y))\\
=&\lim_{j\rightarrow\infty}d(f_1^{kq_{1j}}(x),f_1^{kq_{1j}}(y))\\
=&\lim_{j\rightarrow\infty}d(f_{n_{j_l}+1}^{r_1}\circ f_1^{n_{j_l}}(x),f_{n_{j_l}+1}^{r_1}\circ f_1^{n_{j_l}}(y))\\
=&0,
\end{aligned}
\end{equation}
besides, for the sensitive constant $\delta$, there exists $\delta_1>0$ such that for every $0\leq r<k$ and any $x,y\in X$ with $d(x,y)<\delta_1$, we have $d(f_p^r(x),f_p^r(y))<\delta$.
that is, for every $0\leq r<k$ and any $x,y\in X$, if $d(f_p^r(x),f_p^r(y))\geq\delta$, then $d(x,y)\geq\delta_1$. This together with ({\ref{lymg4}}) can be verified that
\begin{equation}\label{lymg6}
\lim_{j\rightarrow\infty}d(g^{q_{2j}}(x),g^{q_{2j}}(y))\geq\delta_1.
\end{equation}
({\ref{lymg5}}) and ({\ref{lymg6}}) imply that $(X,g)$ is Li-Yorke sensitive.

\end{proof}

\begin{theorem}\label{ymg1}
Let $(X,f_{1,\infty})$ be a $k$-periodic non-autonomous discrete system and $g=f_k\circ f_{k-1}\circ\cdots f_1$. If $(X,f_{1,\infty})$ is cofinitely sensitive, then the autonomous discrete system $(X,g)$ is also cofinitely sensitive. Converse is true when each $f_p(p=1,2,...,k)$ is uniformly continuous.
\end{theorem}

\begin{proof}

Suppose that $(X,f_{1,\infty})$ is cofinitely sensitive, then there is a sensitive constant $\delta>0$ such that for any non-empty open subset $U\subseteq X$, there exist $x,y\in U$ and $N\in \mathbb{N}$ such that for any $n>N$, $d(f_1^n(x),f_1^n(y))>\delta$. Since $g=f_k\circ f_{k-1}\circ\cdots f_1$, then for $m>N$, we have
$$d(g^m(x),g^m(y))=d(f_1^{km}(x),f_1^{km}(y))>\delta.$$
Therefore $(X,g)$ is also cofinitely sensitive.

Conversely,  let $(X,g)$ be cofinitely sensitive with $\delta>0$ as a constant of sensitivity. Then for any non-empty open subset $U\subseteq X$, there exist $x,y\in U$ and $N\in \mathbb{N}$ such that for any $n>N$, $d(g^n(x),g^n(y))>\delta$. Since $(X,f_{1,\infty})$ is $k$-periodic and each $f_p(p=1,2,...,k)$ is uniformly continuous, then for any positive integer $0\leq r<k$, $f_p^r$ is uniformly continuous. Therefore, for the $\delta$ above, there exist $\delta_1$ such that for any $u,v$ with $d(u,v)<\delta_1$ and any integer $0\leq r<k$, we have $d(f_p^r(u),f_p^r(u))<\delta$. that is, for any $u,v$ and any integer $0\leq r<k$, when $d(f_p^r(u),f_p^r(u))>\delta$, then $d(u,v)\geq\delta_1$. Therefore, for any $m>N+1$, notice that $d(g^m(x),g^m(y))=d(f_1^{km}(x),f_1^{km}(y))>\delta$, which implies that $d(f_1^{km-r}(x),f_1^{km-r}(y))\geq\delta_1>\frac{\delta_1}{2}$, where $0\leq r<k$. Take $M=k(N+2)$, then for any $n>M$, we have $d(f_1^n(x),f_1^n(y))>\frac{\delta_1}{2}$. Therefore $(X,f_{1,\infty})$ is cofinitely sensitive.

\end{proof}

\begin{theorem}
Let $(X,f_{1,\infty})$ be a $k$-periodic non-autonomous discrete system and $g=f_k\circ f_{k-1}\circ\cdots f_1$. If the autonomous discrete system $(X,g)$ is syndetically(ergodially, respectively) sensitive, then $(X,f_{1,\infty})$ is also syndetically(ergodially, respectively) sensitive. Converse is true when each $f_p(p=1,2,...,k)$ is uniformly continuous.
\end{theorem}

\begin{proof}

Let $U\subseteq X$ be any nonempty open set. Since $(X,g)$ is syndetically sensitive, there exist $x,y\in U$ and a syndetic sequence $\{n_i\}_{i=1}^{\infty}$ such that $d(g^{n_i}(x),g^{n_i}(y))>\delta$. Thus the sequence $\{kn_i\}_{i=1}^{\infty}$ is also syndetic, then by the fact $g=f_k\circ f_{k-1}\circ\cdots f_1$, one can get that  $d(f_1^{kn_i}(x),f_1^{kn_i}(y))=d(g^{n_i}(x),g^{n_i}(y))>\delta$. Therefore $(X,f_{1,\infty})$ is syndetically sensitive.

 Conversely, suppose that $(X,f_{1,\infty})$ is syndetically sensitive. Thus $N_{f_{1,\infty}}(U,\delta)$ is syndetic. Write $N_{f_{1,\infty}}(U,\delta)$ as an increasing sequence $\{n_i\}_{i=1}^{\infty}$, then there exists an integer $M>0$ such that $n_{i+1}-n_{i}\leq M$ for any $i\in \mathbb{N}$. Let $n_i=km_i+r_i, i\in \mathbb{N}$, where $m_i\in \mathbb{Z_+}, 0\leq r_i<k$. Then $n_{i+1}-n_{i}=km_{i+1}-km_{i}+r_{i+1}-r_{i}\leq M$, which implies that $k(m_{i+1}-m_{i})\leq M+k, \forall i\in \mathbb{N}$. Hence, $\{m_i\}_{i=1}^{\infty}$ is syndetic. In addition, as is shown in Theorem \ref{ymg1}, there exists $\delta_1$ such that for any $i\in \mathbb{N}$, $d(f_1^{kn_i-r_i}(x),f_1^{kn_i-r_i}(y))\geq\delta_1>\frac{\delta_1}{2}$. Hence, $d(g^{m_i}(x),g^{m_i}(y))=d(f_1^{kn_i-r_i}(x),f_1^{kn_i-r_i}(y))>\frac{\delta_1}{2}$. Therefore, $(X,g)$ is syndetically sensitive.

 For ergodially sensitive, the proof is similar to syndetically sensitive and is omitted.

\end{proof}

Let $(X,f_{1,\infty})$ be a $k$-periodic non-autonomous discrete system and $g=f_k\circ f_{k-1}\circ\cdots f_1$. With a similar proof of syndetically sensitive, we can get that if the autonomous discrete system $(X,g)$ is transitive, then $(X,f_{1,\infty})$ is also transitive. However, the converse is incorrect even if $X$ is compact. We will prove it in the following proposition. It is well know that for any positive integer $k\geq2$, there exists a compact autonomous discrete system $(X,f)$ such that $(X,f)$ is transitive but $(X,f^k)$ fails to be transitive {\rm\cite[Example 4]{q4}}.

\begin{proposition}
For any positive integer $k\geq2$, there exists a compact and $k$-periodic non-autonomous discrete system $(X,f_{1,\infty})$ such that  $(X,f_{1,\infty})$ is transitive but $(X,g)$ fails to be transitive, where $g=f_k\circ f_{k-1}\circ\cdots f_1$.
\end{proposition}

\begin{proof}
If the positive integer $k\geq3$, then there exists a autonomous discrete system $(X,f)$ such that $f$ is transitive but $f^{k-1}$ fails to be transitive in example 4 \cite{q4}. Take $$f_{1,\infty}=\{\underbrace{f,f,...,f}_{k-1},id,\underbrace{f,f,...,f}_{k-1},id,...\},$$ then $g=id\circ f^{k-1}=f^{k-1}$. It is easy to see that $(X,f_{1,\infty})$ is $k$-periodic and transitive but $(X,g)$ is not transitive. If the positive integer $k=2$, then there exists a autonomous discrete system $(X,h)$ such that $h$ is transitive $h^2$ fails to be transitive in example 4 \cite{q4}. Put $f_{1,\infty}=\{h^3,h^{-1},h^3,h^{-1},...\}$. Note that $orb(x,{f_{1,\infty}})=orb(x,f)$ and $g=h^2$, so it is easy to see that $(X,f_{1,\infty})$ is $2$-periodic and transitive but $(X,g)$ is not transitive.

\end{proof}

\cite{q29} proved that the mapping sequence $f_{ 1,\infty}=(f_1,f_2,...)$ is (stronger forms)transitive if and only if the mapping sequence $f_{n,\infty}=(f_n,f_{n+1},...) \forall n\in\mathbb{N}$ is also (stronger forms)transitive. Here we extend this result to (stronger forms)sensitivity.

\begin{theorem}\label{fee}
Suppose that each $f_i(i\in \mathbb{N})$ is surjective and $X$ has no isolated point. If $f_{ 1,\infty}$ is sensitive (multi-sensitive, syndetically sensitive, cofinitely sensitive, ergodically sensitive, respectively), then for any $n\geq2 (n\in \mathbb{N})$, $f_{n,\infty}$ is also sensitive(multi-sensitive, syndetically sensitive, cofinitely sensitive, ergodically sensitive, respectively). Converse is true if $f_{ 1,\infty}$ is feeble open.
\end{theorem}

\begin{proof}

 For any nonempty open set $U\subseteq X$ and any $n\geq2 (n\in \mathbb{N})$, since $f_i$ is a continuous surjection for any $i\in \mathbb{N}$, $U_1:=f_1^{-(n-1)}(U)$ is a nonempty open set. Assume that $f_{ 1,\infty}$ is sensitive with the sensitive constant $\delta$, then there exist $x,y\in U_1$ and $m\in \mathbb{N}$ such that $d(f_1^{m}(x),f_1^{m}(y))>\delta$.
And further, there exist $x_0,y_0\in U$ such that $x_0=f_1^{(n-1)}(x)$ and $y_0=f_1^{(n-1)}(y)$. Therefore, we have
$$d(f_n^{m}(x_0),f_n^{m}(y_0))=d(f_n^{m}\circ f_1^{(n-1)}(x),f_n^{m}\circ f_1^{(n-1)}(y))=d(f_1^{m}(x),f_1^{m}(y))>\delta,$$
which shows that $f_{n,\infty}$ is sensitive.

Conversely, for $U\subseteq X$ be any nonempty open set and any $n\geq2 (n\in \mathbb{N})$, since $int(f_i(U))\neq\emptyset$ for any $i\in \mathbb{N}$, there exists nonempty open set $U_1$ such that $U_1\subseteq f_1^n(U)$. And because $f_{n,\infty}$ is sensitive with the sensitive constant $\delta$, then there exist $x,y\in U_1$ and $m\in \mathbb{N}$ such that $d(f_n^{m}(x),f_n^{m}(y))>\delta$. Thus, there exist $x_0,y_0\in U$ such that $x_0=f_1^{-(n-1)}(x)$ and $y_0=f_1^{-(n-1)}(y)$. Therefore, one can get that
$$d(f_1^{m}(x_0),f_1^{m}(y_0))=d(f_n^{m}\circ f_1^{(n-1)}(x_0),f_n^{m}\circ f_1^{(n-1)}(y_0))=d(f_n^{m}(x),f_n^{m}(y))>\delta,$$
which shows that $f_{1,\infty}$ is sensitive.

With a similar argument, the theorem holds for multi-sensitive, syndetically sensitive, cofinitely sensitive, ergodically sensitive, respectively.

\end{proof}

\begin{theorem}\label{fee1}
Let $(X,f_{1,\infty})$ be a surjective, compact, commutative and $k$-periodic non-autonomous discrete system without isolated points. Suppose that
$f_{1,\infty}$ is transitive and that there exists a nonrecurrent point $x_1\in X$, then $f_{1,\infty}$ is sensitive.
\end{theorem}

\begin{proof}

Suppose that $x_1$ is a nonrecurrent point, that is $x_1\notin \overline{\cup_{i=1}^\infty f_1^i(x_1)}$. Then, let $\alpha:=\inf_{y\in\overline{\cup_{i=1}^\infty f_1^i(x_1)}}d(x_1,y)>0$. Since $f_{1,\infty}$ is transitive and compact, there exists a point $x_0\in X$ such that $\overline{orb(x_0)}=X$ by lemma \ref{yinlicd}. Besides, since $X$ has no isolated point, therefore, for any $x\in X$ and any $\delta>0$, there are $m,n\in \mathbb{N}$ with $m-n=kq, q\in\mathbb{N}$ such that $f_1^m(x_0),f_1^n(x_0)\in B(x,\delta)$. Since each $f_i$ is continuous, there exists $\delta_1>0$ such that for any $x'\in B(x_1,\delta_1)$,
\begin{equation}\label{huigui1}
d(f_1^{m-n}(x'),f_1^{m-n}(x_1))<\frac{\alpha}{4}.
\end{equation}
Put $\delta_2:=min\{\delta_1,\frac{\alpha}{4}\}$. Since  $X$ is a nonisolated and $f_{1,\infty}$ commutative surjection, $X\supseteq\overline{orb(f_{1}^n(x_0),f_{1,\infty})}\supseteq f_{1}^n(\overline{orb(x_0,,f_{1,\infty})})=f_{1}^n(X)=X$, which holds that $\overline{orb(f_{1}^n(x_0),f_{1,\infty})}=X$. Therefore, there exists $k\in \mathbb{N}$ such that
\begin{equation}\label{huigui2}
d(f_{1}^{k}\circ f_1^n(x_0),x_1)<\delta_2.
\end{equation}
 Moreover, since $(X,f_{1,\infty})$ is commutative and $k$-periodic,
 \begin{equation}\label{huigui25}
 f_1^k\circ f_1^m=f_1^k\circ f_1^{m-n}\circ f_{m-n+1}^{n}=f_1^k\circ f_1^{m-n}\circ f_{1}^{n}=f_{m-n}\circ f_1^k\circ f_{1}^{n}.
 \end{equation}
 Then,
\begin{equation}\label{huigui3}
\begin{aligned}
&d(f_1^{k}\circ f_1^m(x_0),f_1^{k}\circ f_1^n(x_0))\\
\geq&d(x_1,f_1^{m-n}(x_1))-d(f_1^{k}\circ f_1^n(x_0),x_1)-d(f_1^{k}\circ f_1^m(x_0),f_1^{m-n}(x_1)) \\
>& \alpha-\frac{\alpha}{4}-\frac{\alpha}{4}\\
=&\frac{\alpha}{2}.
\end{aligned}
\end{equation}
In getting the last inequality, we use (\ref{huigui1}), (\ref{huigui2}) and (\ref{huigui25}).

(\ref{huigui3}) yields that either $$d(f_1^{k}\circ f_1^m(x_0),f_{1}^{k}(x))>\frac{\alpha}{4}$$ or $$d(f_1^{k}\circ f_1^n(x_0),f_{1}^{k}(x))>\frac{\alpha}{4}.$$
Therefore $f_{1,\infty}$ is sensitive.
\end{proof}

\begin{definition}
The non-autonomous system $(X, f_{ 1,\infty} )$  is said to be weakly sensitive, if there is $\delta>0$ such that for any nonempty open set $U\subset X$, there exist $x,y\in U$ and $n$ different positive integers $i_k(k=1,2,...,n)$ such that $$d(f_{i_n}\circ \cdot\cdot\cdot f_{i_2}\circ f_{i_1}(x),f_{i_n}\circ \cdot\cdot\cdot f_{i_2}\circ f_{i_1}(y))>\delta.$$

\end{definition}

\begin{definition}
The non-autonomous system $(X, f_{ 1,\infty} )$  is said to be weakly transitive, if for any two non-empty subset $U,V\subseteq X$, there exist $n$ different positive integers $i_k(k=1,2,...,n)$ such that $$f_{i_n}\circ \cdot\cdot\cdot f_{i_2}\circ f_{i_1}(x)(U)\cap f_{i_n}\circ \cdot\cdot\cdot f_{i_2}\circ f_{i_1}(y))\neq\emptyset.$$

\end{definition}

\begin{definition}
The non-autonomous system $(X, f_{ 1,\infty} )$  is said to be weakly Li-Yorke chaos, if there exists an uncountable subset $S\subseteq X$ such that for any different points $x,y\in S$, we have $$\liminf_{n\rightarrow\infty}d(f_{i_n}\circ \cdot\cdot\cdot f_{i_2}\circ f_{i_1}(x),f_{i_n}\circ \cdot\cdot\cdot f_{i_2}\circ f_{i_1}(y))=0$$ and $$\limsup_{n\rightarrow\infty}d(f_{i_n}\circ \cdot\cdot\cdot f_{i_2}\circ f_{i_1}(x),f_{i_n}\circ \cdot\cdot\cdot f_{i_2}\circ f_{i_1}(y))>\delta,$$ where $i_j\neq i_k$ for any $j\neq k$.

\end{definition}

\begin{example}

Let $(X,d)$ be a  metric space and $f:X\rightarrow X$ be sensitive (transitive, Li-Yorke chaotic, respectively). Suppose that $f$ has inverse $f^{-1}$. Let
\begin{equation*}
\begin{aligned}
f_{1,\infty}=&\{f_1,f_2,...f_n,...\}\\
=&\{f,f^{-1},f,f^{-1},f,f^{-1},... \}.
\end{aligned}
\end{equation*}

\end{example}

It is easy to see that $f_{1,\infty}$ is weakly sensitive (transitive, Li-Yorke chaotic, respectively) but not sensitive (transitive, Li-Yorke chaotic, respectively).

\section{Open Questions}

\noindent\textbf{Question 1.} There are many open questions for the relationships between chaos, sensitivity and transitivity. For example, it is know that weakly mixing implies Li-Yorke chaos and that Devaney chaos plus shadowing implies distributional chaos in ADS, whether they are still true in NDS. If not, what condition has to satisfy to ensure they are still true in NDS.

\noindent\textbf{Question 2.}  For ADS, if dynamical property P is preserved under iterations, then for $k$-periodic non-autonomous discrete system, is it true that the ADS $(X,g)$ has property P if and only if $(X,f_{1,\infty})$ has property P, where $g=f_k\circ f_{k-1}\circ\cdots f_1$.

\noindent\textbf{Question 2.}  Wu and Zhu prove that both Li-Yorke chaos and sensitivity of NDS are preserved under iterations when $f_{1,\infty}$ converges  uniformly to $f$ and that neither Li-Yorke chaos nor sensitivity of NDS is iteration invariants if we remove the condition 'converges  uniformly'. Whether these results are still true for weakly Li-Yorke chaos and sensitivity of NDS.



\end{document}